\documentclass[11pt,reqno]{amsart} % use larger type; default would be 10pt

\usepackage[utf8]{inputenc} % set input encoding (not needed with XeLaTeX)

\usepackage{fullpage}
\usepackage{graphicx} % support the \includegraphics command and options
\usepackage{enumerate}
\usepackage{xcolor}
\usepackage{microtype}
\usepackage{amsmath}
\usepackage{amssymb}
\usepackage{amsthm}
\usepackage{amsmath} \usepackage{amssymb}
\usepackage{array} % for better arrays (eg matrices) in maths
\usepackage{comment}
\usepackage{bbm}
\usepackage[hidelinks]{hyperref}
\usepackage{cleveref}

\newcommand{\eqdist}{\stackrel{(d)}{=}}

\numberwithin{equation}{section}
\newtheorem{theorem}{Theorem}[section]
\newtheorem{proposition}[theorem]{Proposition}
\newtheorem{lemma}[theorem]{Lemma}

\theoremstyle{remark}

\theoremstyle{definition}

\def\ber{\color{red}}
\def\eer{\normalcolor}
\def\one{{\mathbbm 1}}
\newcommand{\eps}{\varepsilon}

\def\benr{\begin{enumerate}[label=(\roman*)]}
\def\eenr{\end{enumerate}}

\def\R{\mathbb{R}}

\def\P{\mathbb{P}}
\def\E{\mathbb{E}}

\def\G{\mathbb{G}}

\newcommand{\bma}{\begin{bmatrix}}
\newcommand{\ema}{\end{bmatrix}}

\renewcommand{\hat}{\widehat}
\renewcommand{\tilde}{\widetilde}

\DeclareMathOperator{\Bi}{Bi}
\newcommand{\unn}[2]{[\![#1,#2]\!]}
\renewcommand{\phi}{\varphi}
\newcommand{\mm}{m}

\title{Existence of the free energy for heavy-tailed spin glasses}
\author{Aukosh Jagannath and Patrick Lopatto}
\begin{document}
\begin{abstract}
We study the free energy of a mean-field spin glass whose coupling distribution has power law tails. Under the assumption that the couplings have infinite variance and finite mean, we show that the thermodynamic limit of the quenched free energy exists, and that the free energy is self-averaging.
\end{abstract}

\maketitle
%\setcounter{tocdepth}{1}
%\tableofcontents

\section{Introduction}

We study a mean-field spin glass with heavy-tailed (infinite variance) couplings. This model was introduced  by Cizeau and Bouchaud  30 years ago to understand surprising experimental results on dilute spin glasses with dipolar interactions \cite{cizeau1993mean}. It has been extensively studied by physicists \cite{janzen2010thermodynamics, engel2007replica,  boettcher2014extremal, boettcher2012ground, andresen2011critical,  janzen2010levy, janzen2008replica, janzen2010stability, neri2010phase}, who have also drawn connections to random matrix theory and finance \cite{cizeau1994theory, biroli2007extreme,galluccio1998rational}. 
However, to our knowledge nothing is known rigorously, as the proof techniques developed for lighter-tailed couplings do not apply. 
Indeed, even the existence of the thermodynamic limit of the quenched free energy has not been established. In this note, we prove this existence result and, in addition, prove that the free energy is self-averaging. 
%In this article, we study a spin glass model whose couplings are distributed as a random variable $J$ satisfying $\P\big( | J | > t\big) \sim t^{-\alpha}$ for $\alpha \in (1,2)$; such a variable has finite mean but infinite variance. We show that the limit of the quenched average of the free energy exists, and that the free energy concentrates about this average as the number of spins tends to infinity. 

Let us now be more precise. 
Fix $\alpha \in (0,2)$. Let $J$ be a symmetric random variable such that 
\begin{equation}\label{e:arep}
\P\big( |J| \ge t \big) = \frac{C_0}{t^\alpha}
\end{equation}
for all $|t|>1$, for some constant $C_0 >0$, and such that $\E\big[ | J | \big] < \infty$.
Let $\{J_{ij}\}_{1 \le i < j \le N}$ be a collection of independent, identically distributed random variables with the same distribution as $J$. We consider the Hamiltonian 
\begin{equation}\label{e:hamiltonian}
H(\sigma)  =  \frac{1}{N^{1/\alpha}} \sum_{1\le i < j \le N} J_{ij} \sigma_i \sigma_j,\qquad \sigma \in \Sigma_N = \{ - 1,  + 1\}^N.
\end{equation}
%where $\sigma \in \Sigma_N = \{ - 1,  + 1\}^N$. %We consider only the case $\alpha \in (1,2)$ in this note. 
The partition function and the quenched average of the free energy at inverse temperature $\beta > 0$ are given by
\begin{equation}\label{e:fn}
Z_N(\beta) = \sum_{\sigma\in \Sigma_N} e^{ \beta H(\sigma)}, \qquad F_N(\beta) = \frac{1}{N} \E \big[\log Z_N\big].
\end{equation}

Our first main result establishes that the limit of $F_N(\beta)$ as $N$ grows large exists when $\alpha >1$. (It is straightforward to see that $F_N(\beta)$ is infinite when $\alpha < 1$.)
\begin{theorem}\label{t:main}
For every $\alpha \in (1,2)$ and $\beta >0$, the limit 
$\lim_{N \rightarrow \infty} F_N(\beta)$
exists and is finite.
\end{theorem}
Our second main result establishes that the free energy concentrates around its quenched average (after taking $t = N^{-\delta/2}$).
%a concentration result for the free energy.  In particular, taking $t = N^{-\delta/2}$ in this theorem, we see that the free energy concentrates around its quenched average. 
\begin{theorem}\label{t:concentration}
For every $\alpha \in (1,2)$, $\beta >0$, and $\delta > 0$, there exists a constant $C(\alpha, \beta, \delta, C_0)>1$ such that
\begin{equation}\label{e:concentration}
\P\left( N^{-1} \Big| \log Z_N(\beta)
- \E\big[ \log Z_N(\beta)  \big] \Big| > t   \right) \le \frac{C N^{1 - \alpha + \delta}}{t^2}.
\end{equation}
\end{theorem}

We now comment on the ideas of the proofs, beginning with \Cref{t:main}. It is well known that given a collection of $N$ independent variables with a distribution of the form \eqref{e:arep}, there will be approximately a constant number of them of order $N^{1/\alpha}$, with the rest lower order in $N$ (with high probability). This explains the normalization in \eqref{e:hamiltonian}: after normalizing by $N^{-1/\alpha}$, for each spin $\sigma_i$ there will be a constant number of constant order couplings, with the others much smaller. This configuration of couplings resembles a sparse weighted graph, after neglecting the small couplings.  With this heuristic in mind, we begin by showing that the limit of the quenched free energy for the heavy-tailed spin glass is equal to that of a sparse spin glass model. We then apply a technique from the literature on such models, the combinatorial interpolation of Bayati, Gamarnik, and Tetali \cite{bayati2010combinatorial}, to show that the limit of the quenched free energy exists for this sparse model, which completes the proof of \Cref{t:main}. The proof of \Cref{t:concentration} is through a short martingale  argument, which establishes a bound on the second moment of the difference considered in \eqref{e:concentration}. Then \Cref{t:concentration} follows from an application of Markov's inequality.

%\subsection*{Conventions}
%We will frequently define constants that depend on a parameter. These will be introduced as $C(x)$ for a parameter $x$, and subsequently referred to as $C$ (suppressing the dependence on the parameter in the notation). 
%These constants may change line to line. We typically write $C>1$ for large constants and $c>0$ for small constants. 
%Most constants will depend on $\alpha$, but we suppress this in the notation. Similarly, we suppress all dependencies on $C_0$ and $u_0$.
%
%We write $O\big( f(N) \big)$ for a function $f: \mathbb{Z}_{>0} \rightarrow \mathbb{R}_{> 0}$ to denote a quantity $Q(N)$ such that $Q(N) \le C f(N)$ for some constant $C>1$ and all $N \in \mathbb{Z}_{>0}$. Similarly, we write $o\big( f(N) \big)$ to denote a quantity $q(N)$ such that $\lim_{N\rightarrow \infty} q(N)/f(N)  = 0$.
%
%For the rest of this article, we fix $\beta = 1$ and omit it from the notation, since the value of $\beta$ does not affect our arguments.
%Finally, the function $\log$ always denotes the natural logarithm.

\subsection*{Previous work} 
The existence of the thermodynamic limit of the quenched free energy for a mean field spin glass with Gaussian spins (the Sherrington--Kirkpatrick model) was established by Guerra and Toninelli \cite{guerra2002thermodynamic}. Carmona and Hu showed that the limit exists for identically distributed couplings with finite third moment \cite{carmona2006universality}. Chatterjee then showed that finite variance suffices (which addresses the case of $\alpha > 2$ in \eqref{e:arep}) \cite{chatterjee2005simple}. 
Additionally, Starr and Vermesi  proved a formula for the difference of the free energies of two mean-field spin glass models with infinitely divisible coupling distributions in terms of expectations of multi-spin overlaps \cite{starr2008some}. This class of distributions includes $\alpha$-stable laws, which have infinite variance for $\alpha <2$. At present, it seems difficult to usefully apply this formula in the heavy-tailed context, since we lack information about the overlaps of such models.

\subsection*{Outline} In \Cref{s:reduction}, we reduce the proof of \Cref{t:main} to the proof of an analogous theorem for a certain sparse spin glass model. In \Cref{s:sparse}, we prove \Cref{t:main} assuming an interpolation result, given as \Cref{l:main}. In \Cref{s:lemma4}, we prove \Cref{l:main}, completing the argument for \Cref{t:main}. \Cref{s:concentration} contains the proof of \Cref{t:concentration}.

\subsection*{Acknowledgments} 
The authors thank A. Aggarwal, A. Auffinger, and A. Fribergh for helpful discussions. 
P.L.\ is partially supported by NSF postdoctoral
fellowship DMS-220289. A.J.\ acknowledges the support of the Natural Sciences and Engineering Research Council of Canada (NSERC). Cette recherche a \'et\'e financ\'ee par le Conseil de recherches en sciences naturelles et en g\'enie du Canada (CRSNG),  [RGPIN-2020-04597, DGECR-2020-00199]. 

\subsection*{Notation} For brevity,  we  take $\beta=1$ throughout, though the arguments are identical for $\beta>0$. Further, we let $C,c>0$ denote  constants that may change from line to line, and may depend on $C_0$.
The notation $\unn{a}{b}$ denotes the set of integers $k$ such that $a\le k \le b$.

\section{Reduction to Sparse Hamiltonian}\label{s:reduction}
In this section, we reduce the problem of establishing the limit of the quenched average of the free energy for the Hamiltonian $H$ to the analogous one for a certain sparse Hamiltonian $H^\circ_{u,v,\mm}$ (defined below in \Cref{s:multiedge}), for particular choices of parameters $u,v,\mm$. This reduction proceeds in a series of steps. In \Cref{s:perturbation}, we begin by showing that this problem for $H$ is equivalent to one for a sparse Hamiltonian $\widehat H$, which omits all couplings smaller than a certain threshold. In \Cref{s:fixededge}, we further reduce the problem to studying another sparse Hamiltonian $\widetilde H$, and in \Cref{s:multiedge}, we complete the reduction to $H^\circ_{u,v,\mm}$.

\subsection{The truncated model}\label{s:perturbation}
We begin by viewing $H$ as a perturbation of the Hamiltonian for a model that omits couplings smaller than a certain threshold, and showing 
that the perturbation term provides a negligible contribution to the free energy. 
%when computing the quenched average of the free energy. The argument we apply here is somewhat 
%
%
%
%The argument we use for this purpose is similar to one used in the proof of the Ghirlanda--Guerra identities for mixed $p$-spin models \cite[Section 3.2]{panchenko2013sherrington}.%, which proceeds by introducing a perturbed Hamiltonian \cite{panchenko2013sherrington}. 
To this end, let $\eps>0$ be a parameter that will be chosen later and $R= N^{1/\alpha-\eps}$. We may then write $H(\sigma) = \widehat H(\sigma)  + \hat p(\sigma)$,
where 
\begin{equation}\label{e:perturb}
\widehat H(\sigma) = \frac{1}{N^{1/\alpha}} \sum_{1\le i < j\le N } J_{ij} \one_{| J_{ij} | \ge  R }\sigma_i \sigma_j,\qquad  \hat p(\sigma) = \frac{1}{N^{1/\alpha}} \sum_{1\le i < j \le N} J_{ij}  \one_{| J_{ij} | <  R }\sigma_i \sigma_j.
\end{equation} 
%where $\eps >0$ is parameter that will be chosen later. %We think of $\widehat H(\sigma)$ as the main part of the Hamiltonian and consider $\hat p(\sigma)$ as a perturbation term. 
Denote the partition function and free energy for $\widehat H(\sigma)$ by
$\hat Z_N = \sum_{\sigma\in \Sigma_N} e^{  \hat H(\sigma)}$ and $\hat F_N = \frac{1}{N} \E [\log \hat Z_N]$ respectively.
Our goal is to prove the following:
\begin{lemma}\label{l:l2}
There exists $c> 0$ such that for all $\eps \in (0,c)$, we have that
$\lim_{N\rightarrow \infty} \widehat F_N = \lim_{N\rightarrow \infty}  F_N,$
if the limit on the left exists.
\end{lemma}

The proof of this result will follow by applying the following elementary fact whose proof 
follows by repeatedly applying Jensen's inequality.
\begin{lemma}\label{lem:jensen}
Suppose that $x(\sigma)$ and $y(\sigma)$ are random processes on some finite set $\Sigma$, 
and let  $\langle\cdot\rangle$ denote expectation with respect to the Gibbs measure $\pi(\{\sigma\})\propto \exp(x(\sigma))$.
Then %applying Jensen's inequality twice yields
\begin{equation}\label{eq:jensen}
\E \log \sum_{\sigma \in \Sigma} e^{x(\sigma)} + \E \langle y(\sigma) \rangle \leq \E \log \sum_{\sigma \in \Sigma} e^{x(\sigma)+y(\sigma)} \leq \E \log \sum_{\sigma \in \Sigma} e^{x(\sigma)} + \log\E \langle e^{y(\sigma)}\rangle,
\end{equation}
provided all the expectations are finite.%\footnote{Finiteness is not necessary for this result. We state it this way for simplicity.}
\end{lemma}
\noindent We will apply this result with $x=\hat H$ and $y = \hat p$.

In the following it will be helpful to notice that 
%Note that $\widehat H(\sigma)$ is a sparse model: Heuristically, the probability that a coupling is nonzero in $\widehat H(\sigma)$  is $\P( |J| > R) \approx 1/R^\alpha = N^{\alpha\epsilon -1}$, and since there are $O(N^2)$ possible couplings, a typical realization will have in expectation $O(N^{1+\alpha\epsilon})$ nonzero couplings. We will treat this claim more precisely below.
while $\widehat H(\sigma)$ and $\hat p(\sigma)$ are nominally dependent due to the common coefficients $J_{ij}$, 
we can introduce independence through the following two-step resampling procedure for the $J_{ij}$. 
Let $\{L_{ij}\}_{1 \le i < j \le N}$ be mutually independent random variables such that 
\begin{equation}\label{d:p}
\P( L_{ij} = 1 ) = p,\quad \P( L_{ij} = 0 ) = 1 - p,\quad p  = p_N =  \P( |J| \ge R).
\end{equation}
%%\begin{comment}
%%\begin{definition}\label{d:p}
%%We define a \emph{label} to be an array $ L = \{L_{ij}\}_{1 \le i < j \le N}$ such that each element is equal to $0$ or $1$. Given a collection of random variables $ = \{J_{ij}\}_{1 \le i < j  \le N}$, the corresponding label is defined by $L_{ij} =  \one_{| J_{ij} | \ge  R }$.
%%We say that a label $ L$ is \emph{$ J$-distributed} if the entries $L_{ij}$ are mutually independent, $\P( L_{ij} = 1 ) = p$, and $\P( L_{ij} = 0 ) = 1 - p$, where
%%\begin{equation}\label{d:p}
%%p  = p_N =  \P( |J| \ge R).
%%\end{equation}% for an $\alpha$-stable law $Z$.
%%\end{definition}
%%\end{comment}
Let $\{a_{ij}, b_{ij} \}_{1\le i < j  \le N}$ be a collection of mutually independent random variables (which are also independent from the $L_{ij}$ variables) such that for every interval $I \subset \R$, we have 
\begin{equation}\label{e:arep2}
\begin{aligned}
\P( a_{ij}  \in I ) &= (1-p)^{-1} \P \big( J \in  I \cap \left( -R, R \right) \big)\\
\P( b_{ij} \in I ) &= p^{-1} \P \Big( J \in  I \cap \big( (-\infty, -R ] \cup  [R, \infty )\big) \Big).
\end{aligned}
\end{equation}
%and 
%\begin{equation}\label{e:bdef}
%\end{equation}
Then we have the distributional equalities \begin{equation}\label{e:resampling}
\begin{aligned}
\{J_{ij}\}_{1 \le i < j  \le N} &\eqdist \{ (1- L_{ij} )a_{ij} + L_{ij}b_{ij} \}_{1 \le i < j  \le N}\\
\widehat H(\sigma) \eqdist \frac{1}{N^{1/\alpha}} \sum_{1\le i < j \le N} L_{ij}b_{ij}\sigma_i \sigma_j,
 &\qquad \hat p(\sigma) \eqdist \frac{1}{N^{1/\alpha}} \sum_{1\le i < j \le N} (1- L_{ij} )a_{ij}\sigma_i \sigma_j,
  \end{aligned}
\end{equation}
with the dependence between $\hat H(\sigma)$ and $\hat p(\sigma)$  expressed through the $L_{ij}$.
Observe that after conditioning on $ L$, the sums $\widehat H(\sigma)$ and $\hat p(\sigma)$ are independent. 

Be before proving \Cref{l:l2}, we note the following useful moment bound.
Under the assumption $|x| \le 1$, we have $e^x \le 1 + x + x^2$. %. Using this in conjunction with $1 + x \le e^x$, 
 %\le x + e^{x^2}$. 
%Because $|e^x-1-x|\leq x^2$ for $|x|\leq 1$  
%and $1+x\leq e^x$, there is a universal $C>0$ such that 
Then for any random variable $X$ such that $\E[X] =0$ and $|X|\leq 1$, 
we have 
\begin{equation}\label{eq:one-to-two}
\E \big[\exp  (X) \big]\leq 1 + \E[X] + \E[X^2] \le \exp \big( \E [X^2]\big),
\end{equation}
where we used used $1+ x \le e^x$ in the second inequality.
%With this in hand, we have the 
\begin{lemma}\label{l:expmom}
There exists $C(\eps) >1$ such that
\[
\E \big[ \exp\big( N^{-1/\alpha} (1 - L_{ij}) a_{ij} \big) \big]\le  \exp( C R^{2-\alpha} N^{-2/\alpha} )
\]
for all $1\le i < j \le N$.
\end{lemma}
\begin{proof}
Since $\left| N^{-1/\alpha} (1 - L_{ij}) a_{ij}  \right| \le 1$ by definition, \eqref{eq:one-to-two} yields
%By \eqref{eq:one-to-two}
%By Taylor expansion, there exists a constant $C > 1$ such that $|e^x - 1 -x | < C x^2$ for $|x| \le 1$. By the definitions of $L_{ij}$, $R$, and $a_{ij}$, we have that $\left| N^{-1/\alpha} (1 - L_{ij}) a_{ij}  \right| \le 1$, so
\begin{align}
\E \left[ \exp\left( N^{-1/\alpha} (1 - L_{ij}) a_{ij} \right) \right]
%&\le 
%1 + C \cdot \E \left[  N^{-2/\alpha} (1 - L_{ij})^2 a^2_{ij}  \right] \\
\le  \exp\left( N^{-2/\alpha}   \E \big[  (1 - L_{ij})^2 a^2_{ij}  \big] \right).\label{e:apple}
\end{align}
%where the linear term in the first inequality vanishes due to the symmetry of $a_{ij}$, and we used the bound $1+ x \le e^x$ in the last inequality.
Using \eqref{e:arep2}, $|L_{ij}| \le 1$,  and $p_N = o(1)$ (from \eqref{e:arep}), we have
 \begin{equation}\label{e:apple2}
\E \big[  (1 - L_{ij})^2 a^2_{ij}  \big]
\le 
2\cdot\E\big[ J^2_{ij} \one_{ |J_{ij}| < R} \big]\le C \left(1 + \int_1^R t^{1 - \alpha} \, dt\right)\le C R^{2 -\alpha},
\end{equation}
where the first inequality holds for sufficiently large $N$ (depending on $\alpha$, $\eps$, and $C_0$), and we used \eqref{e:arep} in the second inequality. Inserting \eqref{e:apple2} into \eqref{e:apple} completes the proof.
\end{proof}

%To prove our reduction, let us first note the following elementary fact whose proof follows
%by applying Jensen's inequality repeatedly.

We now turn to our reduction.

\begin{proof}[Proof of Lemma~\ref{l:l2}]
Using the representation \eqref{e:resampling}, we see that if $\langle\cdot\rangle$ denotes
expectation with respect to $\widehat H$, then integrating first in the variables $a_{ij}$ yields
$\E \langle \hat p\rangle =\E \langle\E_a p\rangle=0.$
Thus by Lemma~\ref{lem:jensen}, we have that
\[
\hat{F}_N   \leq F_N \leq \hat F_N +  \log \E \langle \E_a e^{\hat p}\rangle.
\]
For a fixed $\sigma \in \Sigma_N$, the  expectation $\E_a [e^{\hat p(\sigma)}]$ does not depend on the values of the spins $\sigma_i$ because the $a_{ij}$ are symmetric. We then have
\begin{equation}\label{e:imp2}
 \E_a [e^{\hat p(\sigma)}] = \E_a\left[ \prod e^{ N^{-1/\alpha} (1 - L_{ij}) a_{ij}}\right]
 \le \Big( \exp\big( CR^{2-\alpha} N^{-2/\alpha} \big) \Big)^{N^2} = \exp\big( CN^{2 - 2/ \alpha } R^{2-\alpha} \big),
\end{equation} 
where we used the independence of the $(1 - L_{ij}) a_{ij}$ variables and \Cref{l:expmom} for the inequality.
By \eqref{e:arep} and our choice $R = N^{1/\alpha - \epsilon}$, we have  $N^{2 - 2/ \alpha } R^{2-\alpha} = O(N^{1+ \epsilon(\alpha -2 )})$. Using this bound in \eqref{e:imp2} completes the proof.
\end{proof}

\subsection{Model with fixed edge number}\label{s:fixededge}

We next consider a sparse Hamiltonian where the total number of couplings is fixed. Before defining this model, we note the following preliminary lemma. 

Let $M = M_N = \sum_{1\le i <j \le N} L_{ij}$, 
which represents the number of nonzero couplings in $\widehat H(\sigma)$. %We recall that we defined $p = \P( |J| \ge R)$ in \Cref{d:p}. 
Recalling that  $p = C_0 N^{-1 + \alpha\eps}$ by \eqref{d:p} and applying the multiplicative Chernoff inequality yields the following
concentration bound on $M_N$.
\begin{lemma}\label{l:concentrate}
For every $\eps>0$, there exists a constant $c( \eps) >0$ such that 
\begin{equation}\label{e:concentrate}
\left| \E [ M ]  - \frac{C_0}{2} N^{1 + \alpha \eps}  \right| \le c^{-1} N^{\alpha\eps},%\qquad  c N^{1 + \alpha\eps} \le \var\left[  M\right]\le c^{-1} N^{1 + \alpha\eps}
\qquad \P\left( \big| M - \E[ M ]  \big| > c^{-1} N^{1/2 +\alpha\eps/2 + \eps} \right) \le 2 e^{ -  cN^\eps}.
\end{equation}
%%and
%%\begin{equation}\label{e:concentrate}
%%\P\left( \big| M - \E[ M ]  \big| > c^{-1} N^{1/2 +\alpha\eps/2 + \eps} \right) \le 2 \exp\left( -  cN^\eps \right).
%%\end{equation}
\end{lemma}
%\begin{proof}
%By \Cref{d:p}, we have $p = C_0 N^{-1 + \alpha\eps}$, from which the first claim follows. The second claim follows from an application of the 
%multiplicative form of the Chernoff bound.
%\end{proof}

%Set $M = Np$, which represents the mean number of nonzero couplings in the $\widehat H(\sigma)$ Hamiltonian. It is straightforward to check that $M = O(N^{1 + \alpha \eps})$. 

We now introduce a new Hamiltonian $\widetilde H$. Set 
\begin{equation}\label{definitionS}
S_N = \left\lfloor  \frac{C_0}{2} N^{1 + \alpha \eps} \right\rfloor.
\end{equation}
 Let $\{\widetilde L_{ij}\}_{1\le i, j \le N}$ denote the adjacency matrix of a graph drawn uniformly at random from graphs with precisely $S_N$ edges on $N$ vertices, so that each $\widetilde L_{ij}$ is $0$ or $1$. 
% 
%  be 
% $\{0,1\}$-valued random variables whose
%joint distribution is given by sampling uniformly at random from the set of all graphs with precisely $S_N$ edges on $N$ vertices labeled $1$ to $N$, then setting $\widetilde L_{ij} = 1$ if there is an edge between the vertices $i$ and $j$, and $0$ otherwise. 
By definition, exactly $S_N$ of the $\widetilde L_{ij}$ are nonzero.
Set
\begin{equation}\label{e:resampling2}
\{\widetilde J_{ij}\}_{i<j}= \{  \widetilde L_{ij}b_{ij} \}_{ i < j },
\end{equation}
where we recall that the variables $\{b_{ij}\}_{i<j}$ were defined in \eqref{e:arep2}, and we require that the variables $\{\widetilde L_{ij}, a_{ij}, b_{ij}\}_{i<j}$ are all mutually independent. Then define the Hamiltonian
\[
\tilde H(\sigma) =  \frac{1}{N^{1/\alpha}} \sum_{1\le i < j \le N} \tilde J_{ij} \sigma_i \sigma_j,
\]
and set $\tilde Z_N = \sum_{\sigma\in \Sigma_N} e^{  \tilde H(\sigma)}$  and $ \tilde F_N = \frac{1}{N} \E \big[\log \tilde Z_N\big].$
We now show that  $\widetilde F$  is a good approximation to  $\widehat F$.

%\begin{lemma}\label{l:fixededge}
%There exists $c(\alpha)> 0$ such that, if $c(\alpha) > \eps$, then the following holds. If the free energy $\widetilde F_N$ for $\widetilde H(\sigma)$ converges to a  limit, then the free energy $\widehat F_N$ for $\widehat H(\sigma)$ converges to the same limit.
%\end{lemma}
\begin{lemma}\label{l:fixededge}
There exists $c> 0$ such that the following holds for all $\eps \in (0,c)$. We have 
$\lim_{N\rightarrow \infty} \widetilde F_N = \lim_{N\rightarrow \infty}  \widehat F_N,$
if the limit on the left exists.
\end{lemma}
\begin{proof}
Given a realization of the $\{ \tilde J_{ij} \}$, we consider the effect on $\widetilde Z$ of fixing some $(i,j)$ and changing $\tilde L_{ij}$ to $1 - \tilde L_{ij}$ (that is, adding or deleting a coupling). If we denote the partition function that results from this change by $\widetilde Z_{\mathrm{new}}$, we see by direct calculation that 
\begin{equation}\label{e:edgeremoval}
| \log \widetilde Z  - \log \widetilde Z_{\mathrm{new}} | \le N^{-1/\alpha} |b_{ij}|.
\end{equation}
Using \eqref{e:arep2}, we find $\E\big[ | b_{ij} | \big]  = O( N^{1/\alpha - \eps} )$, so \eqref{e:edgeremoval} yields
\begin{equation}\label{e:add}
\big| \E [\log \widetilde Z] -  \E [\log \widetilde Z_{\mathrm{new}} ]\big| \le C N^{-\eps}.
\end{equation}
for some constant $C>0$. Therefore, adding or removing $k$ couplings in this way results in a change of at most  $C k$ to the expected log-partition function. 

Let $\mathcal A$ be the event on which  $\left| M - S  \right| \le c^{-1} N^{1/2 +\alpha\eps/2 + \eps}$ holds. We write 
\begin{align}\label{e:c1}
| \hat F_N - \tilde F_N | \le 
\frac{1}{N} \Big| \E \big[ (\log \hat Z - \log \tilde Z) \one_{\mathcal A} \big] \Big|
+ \frac{1}{N} \Big| \E \big[  (\log \hat Z - \log \tilde Z) \one_{\mathcal A^c}  \big] \Big|.
\end{align}

We now consider an alternative sampling scheme for $\hat H$. Begin by sampling the mutually independent variables $\{ L_{ij}, \tilde L_{ij}, a_{ij}, b_{ij} \}$, and let $\ell$ be the (random) number of nonzero $L_{ij}$. From these realizations, we obtain a realization of $\tilde H$; we will use this realization of $\tilde H$ to produce a coupled realization of $\hat H$ in the following way. 
If $\ell > S$, choose uniformly at random $\ell - S$ distinct index pairs $\mathcal P = \big\{ (i_a,j_a) \big\}_{a=1}^{\ell -S}$ 
from the set $\big\{(i,j) : \tilde L_{ij} =0\big\}$. 
%such that $\tilde L_{i_a j_a} =0$ for each $a$. 
For $1 \le i < j \le N$, we set $\hat L_{ij} = \tilde L_{ij}$ if $(i,j) \notin \mathcal P$, and $\hat L_{ij} = 1 - \tilde L_{ij}$ if $(i,j) \in \mathcal P$. Likewise, if $\ell < S$, we choose uniformly at random $S - \ell$ index pairs $\mathcal P = \big\{ (i_a,j_a) \big\}_{a=1}^{\ell -S}$  
from the set $\big\{(i,j) : \tilde L_{ij} =1\big\}$;
%such that $\tilde L_{i_a j_a} = 1$ for each $a$, 
we set $\hat L_{ij} = \tilde L_{ij}$ if $(i,j) \notin \mathcal P$, and $\hat L_{ij} = 1 - \tilde L_{ij}$ if $(i,j) \in \mathcal P$. Finally, if $\ell = S$, we set $\hat L_{ij} = \tilde L_{ij}$ for all $(i,j)$ such that $1 \le i < j \le N$. 

Since this procedure is symmetric with the respect to the edges $(i,j)$, and $\ell$ equals the number of nonzero $\hat L_{ij}$ labels, we find that 
\begin{equation*}
\hat H(\sigma)  \eqdist \frac{1}{N^{1/\alpha}} \sum_{i < j} \hat L_{ij} b_{ij}  \sigma_i \sigma_j,
\end{equation*}
and we have produced a coupling between $\hat H(\sigma)$ and $\tilde H(\sigma)$ through the addition or subtraction of $|\ell - S|$ couplings from $\tilde H(\sigma)$.

On the event $\mathcal A$, we have $|\ell - S| \le c^{-1} N^{1/2 +\alpha\eps/2 + \eps}$ by definition. Therefore we obtain from  \eqref{e:add} that 
\begin{equation}\label{e:c2}
 \Big| \E\big [ (\log \hat Z - \log \tilde Z\big) \one_{\mathcal A} \big] \Big|
\le C N^{1/2 +\alpha\eps/2 + \eps}.
\end{equation}
For the other term in \eqref{e:c1}, we use H\"older's inequality to show that
\begin{equation*}
 \Big| \E \big[  (\log \hat Z - \log \tilde Z) \one_{\mathcal A^c}  \big] \Big|
 \le \P(\mathcal A^c)^{\frac{\eps}{1 + \eps}}
 \E \big[  |\log \hat Z \big|^{1+\eps}   \big]^{\frac{1}{1+\eps}} +
 \P(\mathcal A^c)^{\frac{\eps}{1 + \eps}} \E \big[  |\log \tilde Z |^{1+\eps}   \big]^{\frac{1}{1+\eps}}.
\end{equation*}
We give details only for the bound on $\E \big[  |\log \hat Z |^{1 +\eps}   \big]$, as the bound for $ \E \big[  |\log \tilde Z |^{1 +\eps}  \big]$ is similar. 

By definition, we have $| \hat H (\sigma) | \le N^{ - 1/\alpha} \sum_{1\le i<j \le N} |b_{ij}|,$
which implies
\begin{equation*}
2^N \exp\left( - N^{ - 1/\alpha} \textstyle\sum_{i<j} |b_{ij}|  \right) \le \widehat Z_N  \le 2^N \exp\big( N^{ - 1/\alpha}  \textstyle\sum_{i<j} |b_{ij}|  \big).
\end{equation*} 
This in turn implies that
$\big|\log \hat Z \big| \le N \log 2 + N^{ - 1/\alpha} \sum_{i<j} |b_{ij}|.$
Therefore
\begin{align}
\E \big[  |\log \hat Z \big|^{1 +\eps}   \big]
\le \E \big[ ( N \log 2 +  \textstyle\sum |b_{ij}| )^{1 +\eps} \big]\le C N^{2\eps} \E\big[( N^{1+\eps} +  \textstyle\sum |b_{ij}|^{1+\eps} ) \big]\le C N^{5},\label{conc2}
\end{align}
for some constant $C(\eps)>1$. 
In the second inequality, we used the elementary inequality
$\left(\sum_{i=1}^k x_i \right)^{1+\eps}
\le
k^\eps  \sum_{i=1}^k |x_i|^{1+\eps} ,$
which follows from H\"older's inequality.
In the third inequality, we used that the $(1+\eps)$-th moment of $|b_{ij}|$ exists for $\eps$ small enough (depending only on $\alpha > 1$) and is less than $CN^2$ for some $C( \eps)>1$. We conclude using \eqref{e:concentrate} and \eqref{conc2} that there exists $c( \eps) > 0$ such that
\begin{equation}\label{e:c3}
\P(\mathcal A^c)^{\frac{\eps}{1 + \eps}}
 \E \big[  |\log \hat Z |^{1+\eps}   \big]^{\frac{1}{1+\eps}} \le c^{-1} N^5 \exp( -  cN^\eps ).
\end{equation}
The conclusion now follows from combining \eqref{e:c1}, \eqref{e:c2}, and \eqref{e:c3}.
\end{proof}

\subsection{Model with multi-edges}\label{s:multiedge}

The Hamiltonian $\tilde H$ may be thought of as arising from a collection of $S_N$ weighted edges on a simple random graph. However, it will be  convenient to consider instead a similar model where the edges are sampled with replacement from pairs of vertices $\big\{(i,j)\big\}_{1 \le i \le j \le N}$. In particular, self-edges of the form $(i,i)$ are allowed, as well as multi-edges, meaning that an edge $(i,j)$ may appear two or more times.

We define the Hamiltonian $H^\circ_{u,v,\mm}(\sigma)$ as follows. The parameters $u,v,\mm$ denote the edge number, vertex number, and effective normalization, respectively, as we now explain. % to vary, respectively. This flexibility will be useful in the next section.
Let $\{ c_a \}_{1 \le a \le u }$ be a collection of identically distributed, independent random variables, where each $c_a$ is an ordered pair $(i,j)$ chosen uniformly at random from the set $\big\{ ( i,j) \big\}_{1 \le i \le j \le v}$. We write $c_a(1)$ and $c_a(2)$ for the first and second coordinates of $c_a$, respectively. While the definition of the $c_a$ variables depends on both $u$ and $v$, we omit this from the notation.

We set $q = \P\big( | J | \ge \mm^{1/\alpha -\eps} \big)$ and define identically distributed, independent random variables $\{ d_{a}\}_{1 \le a \le u} =\big\{ d_{a}(\mm)\big\}_{1 \le a \le u}$ by requiring that 
\begin{equation}\label{e:defd}
\P( d_a \in I) = 
q^{-1} \P\Big( J \in I \cap \big(  (-\infty, - \mm^{1/\alpha -\eps}] \cup [ \mm^{1/\alpha -\eps}, \infty) \big) \Big)
\end{equation}
for every interval $I \subset \R$. We define 
\begin{equation*}
 H^\circ_{u,v,\mm}(\sigma) =  \frac{1}{\mm^{1/\alpha}}\sum_{1 \le a \le u} d_a \sigma_{c_a(1)} \sigma_{c_a(2)},
 \end{equation*}
and set $Z^\circ_v(u,\mm) = \sum_{\sigma\in \Sigma_N} \exp\big( H^\circ_{u,v,\mm}(\sigma)\big)$ and  $ F^\circ_{u,v,\mm} = \frac{1}{\mm} \E \big[\log  Z^\circ_v(u,m) \big]$.

We now define a quantity that represents the total number of loops and multi-edges among the $c_a$. We set
\[
f_{u,v} =\sum_{1\le i  \le v}    
\sum_{1\le a \le u} \one_{c_a = (i,i)} +
\sum_{1\le i < j \le v} \one\Bigg\{ \Bigg(\sum_{1\le a \le u}  \one_{c_a = (i,j)}\Bigg) > 1 \Bigg\}\cdot \Bigg( 
\sum_{1\le a \le u} \one_{c_a = (i,j)}  - 1\Bigg).
\]
\begin{lemma}
For every $\eps > 0$, there exists $c(\eps)>0$ such that
\begin{equation}\label{e:multic}
\P ( f_{S_N,N} \ge  N^{3 \alpha \eps} ) 
\le  \exp( - c N^\eps ).
\end{equation}
\end{lemma}
\begin{proof}
Consider a sequence of random multi-graphs $\mathbb G_t$ on the vertex set $\{1,\dots, N\}$ built in the following way. Let $\mathbb G_0$ be the graph with no edges, and for each $t \in \unn{1}{S_N}$, let $\mathbb G_t$ be the graph with edge set $\{c_a\}_{1\le a \le t}$. 
The definition of $\mathbb G_t$ naturally extends to $t> S_N$ by choosing additional multi-edges uniformly at random. 
Let 
$T = \min \big\{ t : 
|\mathcal E_t | \ge S_N -  N^{3 \alpha \eps} \big\},$
 where 
\[
\mathcal E_t =
\big\{
(i,j) : 1 \le i < j \le N, \, (i,j) = c_a \text{ for some } a \le t
\big\}
\]
is the set of (non-loop) edges that have been added at time $t$ (after removing duplicates). Observe that 
\[
\P ( f_{S_N,N} \ge N^{3 \alpha \eps} ) 
= \P( T \le S_N),
\]
so it suffices to bound $\P( T \le S_N)$.

Let $T_1 = \min \{ t : 
|\mathcal E_t | \ge 1
\}$ and define $T_i$ for $i \ge 2$ by
\[
T_{i} 
= \min \{ t : 
|\mathcal E_t | \ge i
\} - \min \{ t : 
|\mathcal E_t | \ge i-1
\}.
\]
Then, by definition,
\begin{equation}\label{Tsum}
T = \sum_{i=1}^{S_N - N^{3\alpha\eps}}
T_i.
\end{equation}
We say that $X$ is a geometric random variable with parameter $p$ if $\P(X=k) = (1-p)^{k-1} p$ for all $k \ge 1$; then $\E[X] = p^{-1}$. We observe that each $T_i$ in \eqref{Tsum} is a geometric random variable with parameter
\[
p_i = \frac{A - N - i + 1}{A}
\]
where $A = N(N+1)/2$ is the number of possible edges and loops on a graph with $N$ vertices. Using this representation, \eqref{Tsum}, and an integral approximation, we compute that 
\begin{align}
\E[T]
=
A \sum_{i=1}^{S_N - N^{3\alpha\eps}}%\nonumber
\frac{1}{A - N - i + 1}
%&\ge A \int_{2}^{S_N - N^{3\alpha\eps} - 1} \frac{1}{A - N - x - 1}\,dx\\ 
\ge A \log\left(
\frac{A-N }{A-N - S_N + N^{3\alpha\eps} + 2}\right)%\nonumber\\
%&\ge A \cdot \frac{S_N - N^{3\alpha\eps} - 2}{A-N} 
\ge S_N - 2 N^{3\alpha\eps}\label{etlast}
\end{align}
for sufficiently large $N$ (depending on $\eps$), where we used 
$\log\big(1/(1-x)\big) \ge x$
in \eqref{etlast}. 
Defining $T^\circ = T - \E[T]$, we bound
\begin{align}
\P( T \le S_N) = 
\P\big( T^\circ \le S_N - \E[T]\big) \le \P\left( T^\circ\le - 2 N^{3\alpha\eps} \right)
%=\P\left( e^{-T^\circ} \ge e^{ 2 N^{3\alpha\eps} } \right)
\le e^{ - 2 N^{3\alpha\eps} } \E \big[  e^{-T^\circ} \big].\label{Tcircreduce}
\end{align}
Then to complete the proof, it suffices to bound the right side of \eqref{Tcircreduce}.

Writing $T_i^\circ = T^\circ_i - \E\big[ T_i^\circ \big]$, we have
\begin{align}
\E \big[  e^{-T^\circ} \big]
= \prod_{i=1}^{S_N - N^{3\alpha\eps}  }
\E \big[  e^{-{T_i^\circ}} \big] 
&= \prod_{i=1}^{S_N - N^{3\alpha\eps} }
\frac{e^{1/p_i}}{1 + p_i^{-1}( e -1 )}
\le 
%\prod_{i=1}^{S_N - N^{3\alpha\eps} }
%e^{ (p_i^{-1} - 1)}
%\\
%& = 
\exp 
\sum_{i=1}^{S_N - N^{3\alpha\eps} }
( p_i^{-1} - 1 )
.\label{picalc}
\end{align}
We compute
\begin{equation}
 p_i^{-1} -1  =  \frac{A}{A - N - i + 1} -1
= 
\frac{N + i - 1 }{A - N - i + 1}
\le 
\frac{N + S_N - 1 }{A - N - S_N + 1}
\le C_1 N^{-2} S_N \label{previneq}
\end{equation}
for sufficiently large $N$, and a constant $C_1 > 1$. 
Inserting \eqref{previneq} into 
\eqref{picalc} and \eqref{Tcircreduce}, we find
\begin{equation*}
\P(T\le S_N) 
\le 
\exp\left( -2 N^{3 \alpha \eps} +  C_1 N^{-2} S_N^2\right)
\le
\exp\left( -2  N^{3 \alpha \eps} +  C_1 N^{-2} S_N^2\right)
\le \exp( -c N^{\eps}).
\end{equation*}
This completes the proof.
\end{proof}

\begin{lemma}\label{l:multimodel}
There exists $c> 0$ such that the following holds for all $\eps \in (0,c)$. We have 
$\lim_{N\rightarrow \infty} F^\circ_{S_N,N,N} = \lim_{N\rightarrow \infty}  \widetilde  F_N,$
if the limit on the left exists.
\end{lemma}
\begin{proof}
\begin{comment}
We first identify the number of loops and multi-edges among the $c_a$. For each $a$, $c_a$ is one of $N(N+1)/2$ possible edges. For any $(i,j)$ with $i\le j$, let 

be the number of times the edge $(i,j)$ is duplicated, and let $f = \sum_{ i \le j} f_{ij}$ be the total number of duplicates. %Though the $f_{ij}$ are dependent, this dependence is weak,
We have 
\begin{equation}\label{peach}
\E\big[f^k\big]\le C \sum_{ i \le j} \E\big[f^k_{ij}\big]
\le C N^2 \E\big[f^k_{11}\big] \le C N^2 \E\big[f^2_{12}\big], 
\end{equation}
for some $C(k) > 0$, using that the $f_{ij}$ are identically distributed and $|f_{11}| \le 1$. Using the definition of $S_N$ in \eqref{definitionS}, we compute $\E \big [f^2_{11} \big] \le C N^{-2 + 2 \alpha \eps}$. Inserting this into \eqref{peach} gives
\begin{equation}
\E \big[ f^k \big] \le C(k) N^{2 \alpha \eps}
\end{equation}
for any $k \ge 0$, where $C(k,  \eps) >0$ is a constant. Then, by Markov's inequality, we have
%\begin{equation}
%\P\left( f \ge t \right) \le %\frac{ C(k) N^{2 \alpha %\eps}}{t^k}.
%\end{equation}
%Choosing $k$ large enough in a %way that depends only on %$\alpha$ and $\eps$, we find %that 
%\begin{equation}\label{e:multic}
%\P\left( f \ge N^{3\alpha\eps} %\right) \le C N^{-100}.
%\end{equation}
\end{comment}

Given \eqref{e:multic}, the rest of the proof proceeds similarly to the proof of \Cref{l:fixededge}. We indicate only the main points here. Let 
\[
D 
= \big\{ a \in \{1,2,\dots, S\} :
c_a =c_b \text{ for some } b<a, \text{ or $c_a = (i,i)$ for some $i\in \unn{1}{N}$}
\big\}
\]
denote the set of duplicate edges and loops. As described below \eqref{e:c1}, to sample from the Hamiltonian $\tilde H$, one may begin with $H^\circ_{S_N,N,N}$, remove all edges $c_a$ such that $a\in D$, and restore $|D|$ edges chosen uniformly at random from all size $|D|$ subsets of distinct edges contained in
\[
\{ (i,j) : 1 \le i < j \le N \} \setminus \{c_1, c_2, \dots, c_S \},
\]
where $\{c_1, c_2, \dots, c_S \}$ denotes the original set of edges in  $H^\circ_{S,N,N}$.

As shown in \eqref{e:add}, the addition or deletion of some edge $c_a$ with weight $d_a(N)$ in the Hamiltonian $H^\circ_{S_N,N, N}$ changes $\E \big[\log Z_N^\circ(S,N) \big]$ by at most a constant $C>1$. Let $\mathcal A$ denote the event where $f \le N^{3\alpha\eps}$ holds. Then, under the coupling given in the previous paragraph, we have
\begin{equation}\label{e:c22}
 \bigg| \E \Big[ \big(\log \tilde  Z - \log  Z^\circ_N(S,N) \big) \one_{\mathcal A} \Big] \bigg|
\le C N^{3 \alpha \eps}.
\end{equation}
Further, by a computation nearly identical to \eqref{conc2}, we have 
\begin{equation}\label{peach2}
\E\Big[\big|\log \tilde Z \big|^{1+\eps}\Big] + \E\Big[\big|\log Z^\circ_N(S_N,N) \big|^{1+\eps}\Big]  \le
CN^5.
\end{equation}
H\"older's inequality shows that
\begin{align}
 \bigg| \E \Big[  \big(\log  Z^\circ_N(S,N)  - \log \tilde Z\big) \one_{\mathcal A^c}  \Big] \bigg|
 \le \P(\mathcal A^c)^{\frac{\eps}{1 + \eps}}
 \E \Big[  \big| \log Z^\circ_N(S,N)  \big|^{1+\eps}   \Big]^{\frac{1}{1+\eps}} +
 \P(\mathcal A^c)^{\frac{\eps}{1 + \eps}} \E \Big[  \big|\log \tilde Z \big|^{1+\eps}   \Big]^{\frac{1}{1+\eps}}.\label{peach3}
\end{align}
The conclusion now follows from combining \eqref{e:multic}, \eqref{peach2}, and \eqref{peach3}.
\end{proof}
\section{Free Energy for Sparse Hamiltonian}\label{s:sparse}
In this section we apply the combinatorial interpolation strategy of \cite{bayati2010combinatorial} to show the existence of the limit of the free energy for the multi-edge model.
%and we use the shorthand $Z^\circ_v(u) = Z^\circ_v(u,N)$.

We begin by stating two preliminary lemmas. The first is our main interpolation result.   We prove it at the end of \Cref{s:lemma4}, below. In its statement, the notation $\Bi(n,p)$ denotes a binomial random variable with $n$ trials and success probability $p$, and we recall $S_N$ was defined in \eqref{definitionS}.  We also recall $Z^\circ_v(u, m)$ is the partition function for the Hamiltonian $H^\circ_{u,v,\mm}$ defined in \Cref{s:multiedge}.
\begin{lemma}\label{l:main} For every $1\le N_1, N_2 \le N$ such that $N_1 + N_2 = N$, we have
\begin{equation}\label{e:main}
\E\big[ \log Z^\circ_N (S_N,N)\big] \ge \E \big[\log Z^\circ_{N_1} (\mathcal M_1,N)\big] + \E \big[\log Z^\circ_{N_2} (\mathcal M_2,N)\big],
\end{equation}
where $\mathcal M_1$ is distributed as $\Bi(S_N, N_1/N)$ and $\mathcal M_2 = S_N - \mathcal M_1$ is distributed as $\Bi(S_N, N_2/N)$. Here $\mathcal M_1$ and $\mathcal M_2$ are independent of the random variables in the definition of $H^\circ$.
\end{lemma}
We also require the following sub-additivity lemma.

\begin{lemma}[{\cite[Theorem 23]{de1952some}}]\label{l:subadd}
Suppose that  the sequence $(a_N)_{N=1}^\infty$ satisfies 
${a_N \le a_{N_1} + a_{N - N_1} + \varphi(N)}$
for all $N, N_1$ such that $N/3 \le N_1 \le 2N/3$, for some positive increasing function $\varphi$ with $\int_1^\infty  \varphi(t) t^{-2}\,dt  < \infty.$ 
Then $\lim_{N\rightarrow \infty} a_N/N =  L$ for some $L$ such that ${ -\infty \le  L  < \infty}$. 
\end{lemma}

\subsection{Proof of \Cref{t:main}}
Let $1\le N_1, N_2 \le N$ be parameters. We define 
\begin{equation*}
S^{(1)} = \frac{S_N N_1}{N}, \qquad S^{(2)} = \frac{S_N N_2}{N}.
\end{equation*}
\begin{lemma}\label{l:newapprox}
Fix $\eps >0$. Then there exists $C(\eps) > 0$ such that the following holds.
For every $1\le N_1, N_2 \le N$ such that $N/3 \le N_1 \le 2N/3$ and $N_1 + N_2 =N$, we have
\begin{align}\label{m1swap}
\Big|
\E \big[\log Z^\circ_{N_1} (\mathcal M_1,N)\big]
-
 \E \big[\log Z_{N_1}^\circ (S^{(1)},N ) \big]
 \Big|
 &\le CN^{2/3}\\
\Big|
\E \big[\log Z^\circ_{N_2} (\mathcal M_2,N)\big]
-
 \E \big[\log Z_{N_2}^\circ (S^{(2)},N ) \big]
 \Big|
& \le CN^{2/3}. \label{m2swap}
\end{align}
\end{lemma}
\begin{proof}
We prove only \eqref{m1swap}, since the proof of \eqref{m2swap} is similar.
Recall the definition of $\mathcal M_1$ from \eqref{l:main}, and note that  $S^{(1)} = \E[ \mathcal M_1]$. By the Chernoff bound, there exists a constant $c( \eps) >0$ such that 
\begin{equation}\label{e:concentrate2}
\P\left( \mathcal A_1^c \right) \le \exp\left( -  cN^\eps \right),
\quad
\mathcal A_1 = 
\big\{
| \mathcal M_1 - S^{(1)} | \le c^{-1} N_1^{1/2} N^{\alpha\eps/2 + \eps}
\big\}.
\end{equation}
%An similar statement holds for the analogously defined event $\mathcal A_2$. 
Similarly to the coupling given below \eqref{e:multic}, we may couple the Hamiltonians $H^\circ_{\mathcal M_1, N_1, N}$ and  $H^\circ_{S^{(1)},N_1,N }$ by deleting or adding $| \mathcal M_1 - S^{(1)} |$ edges $c_a$ from $H^\circ_{\mathcal M_1, N_1, N}$, where each member of the set of modified edges is chosen uniformly at random from the set of all possible edges. By \eqref{e:add} and the definition of $\mathcal A_1$, we have
\begin{equation}\label{pepper1}
 \bigg| \E \Big[ \big(\log  Z^\circ_{N_1}(\mathcal M_1,N) - \log  Z^\circ_{N_1}(S^{(1)},N) \big) \one_{\mathcal A_1} \Big] \bigg|
\le C  N^{1/2 + \alpha\eps/2 + \eps}.
\end{equation}
Further, arguing similarly to \eqref{conc2}, we find
\begin{equation}\label{pepper2}
\E\Big[\big|\log Z^\circ_{N_1}(\mathcal M_1,N) \big|^{1+\eps}\Big] + \E\Big[\big|\log Z^\circ_{N_1}(S^{(1)},N) \big|^{1+\eps}\Big]  \le
CN^5.
\end{equation}
Then applying H\"older's inequality and using \eqref{e:concentrate2} (as in \eqref{peach3}) shows 
\[
 \bigg| \E \Big[ \big(\log  Z^\circ_{N_1}(\mathcal M_1,N) - \log  Z^\circ_{N_1}(S^{(1)},N) \big) \one_{\mathcal A_1^c} \Big] \bigg|
 \le C N^5 \exp( -c N^\eps).
\]
Combining the previous line with \eqref{pepper1} completes the proof.
\end{proof}
Recall that $\eps>0$ is a parameter. We write 
$H^\circ _{S^{(1)}, N_1 , N } = \hat H^\circ + \hat p^\circ,$
where we define 
\begin{align}\label{hatHcirc}
\hat H^\circ
&=
\frac{1}{N^{1/\alpha}} \sum_{1 \le a \le S^{(1)}} d_a(N) \sigma_{c_a(1)} \sigma_{c_a(2)} \one_{\big|N^{-1/\alpha} d_a(N)\big| \ge N_1^{-\eps} },\\
\hat p^\circ &= \frac{1}{N^{1/\alpha}}\sum_{1 \le a \le S^{(1)}} d_a(N) \sigma_{c_a(1)} \sigma_{c_a(2)}\one_{\big|N^{-1/\alpha} d_a(N)\big| < N_1^{-\eps} }.\label{hatHcirc2}
\end{align}
We also define
$\hat Z^\circ_N  = \sum_{\sigma\in \Sigma_N} \exp\big( \hat H^\circ (\sigma)\big).$
Set 
\begin{equation}\label{p1def}
p^{(1)}= p^{(1)}_N  = \P\Big( \big|N^{-1/\alpha} d_1(N)\big| \ge N_1^{-\eps} \Big) = \frac{N_1^{\alpha\eps} }{N^{\alpha\eps}},
\end{equation}
where the last equality follows from the definition of $d_a$ and \eqref{e:arep}, and we supposed that $N$ is sufficiently large (depending on $C_0$).
We let $\{x_a, y_a \}_{1\le a \le S^{(1)}}$ be a collection of mutually independent random variables such that for every interval $I \subset \R$, we have 
\begin{align}\label{e:arep22}
\P( x_a \in I ) &= \big(1-p^{(1)}\big)^{-1} \P \big( N^{-1/\alpha} d_1(N) \in  I \cap ( -N_1^{-\eps}, N_1^{-\eps} ) \big)\\
\P( y_a \in I )&= \big(p^{(1)}\big)^{-1} \P \Big( N^{-1/\alpha} d_1(N) \in  I \cap \big( (-\infty, -N_1^{-\eps} ] \cup  [N_1^{-\eps}, \infty )\big) \Big).\label{e:bdef2}
\end{align}
Let $L^\circ = \{L^\circ_a\}_{1 \le a \le S^{(1)} }$ be a collection of independent, identically distributed random variables such that $\P(L_a = 1) = p^{(1)}$ and $\P(L_a = 1) = 1- p^{(1)}$. We further impose the condition that the collection $\{L^\circ_a, x_a,y_a \}_{1 \le a \le S^{(1)} }$ is mutually independent.

With these definitions, we have the distributional equalities
\begin{equation}\label{e:resampling23}
\big\{N^{-1/\alpha} d_a(N) \big\}_{1 \le a \le S^{(1)} }\eqdist \{ (1- L^\circ_a )x_a + L^\circ_a y_a \}_{1 \le a \le S^{(1)}},
\end{equation}
\begin{equation}\label{lemon}
\widehat H^\circ(\sigma) \eqdist  \sum_{1 \le a \le S^{(1)} } L^\circ_a y_a\sigma_{c_a(1)} \sigma_{c_a(2)},\qquad  \hat p^\circ(\sigma) \eqdist  \sum_{1 \le a \le S^{(1)}} (1- L^\circ_a  )x_a\sigma_{c_a(1)} \sigma_{c_a(2)},
\end{equation}
with the dependence between $\hat H^\circ(\sigma)$ and $\hat p^\circ(\sigma)$  expressed through the $L^\circ_a$.
We observe that after conditioning on $ L^\circ$, the sums $\widehat H^\circ(\sigma)$ and $\hat p^\circ(\sigma)$ are independent.

\begin{lemma}\label{l:expmom2}
Fix $\eps >0$. Then there exists $C(\eps) > 0$ such that the following holds.
For every $1\le N_1 \le N$ such that $N/3 \le N_1 \le 2N/3$, we have
\[
\E \big[ \exp \big(  (1 - L_a) x_a \big) \big]\le \exp( C N^{-2\eps} ).
\]
\end{lemma}
\begin{proof}
By \eqref{eq:one-to-two},
\begin{align}
\E \big[ \exp\big(  (1 - L^\circ_a) x_a \big) \big]
%&\le 
%1 + C\cdot\E \left[  (1 - L_a^\circ)^2 x_a^2  \right] \\
\le {C}  \exp( C\cdot  \E [  (1 - L_a^\circ)^2 x_a^2  ] ).\label{e:apple12}
\end{align}

Next, using the definition of $d_a$ from \eqref{e:defd}, we have 
\begin{equation}\label{e:rescaletail2}
\P\Big(\big|N^{-1/\alpha} d_a(N)\big| > t \Big)  = \frac{1}{N^{\alpha\eps} t^\alpha}
\end{equation}
for $t \ge N^{-\eps}$. 
Then using $|L^\circ_{ij}| \le 1$, $\big(1 - p^{(1)}\big)^{-1} \le 3$ (from \eqref{e:arep} and the assumption on $N_1$), and the definition \eqref{e:arep22}, we have
 \begin{align}\label{e:apple22}
\E \big[  (1 - L^\circ_a)^2 x^2_a  \big]
\le 3  \E\Big[ \big|N^{-1/\alpha} d_a(N)\big|^2 \one_{\big|N^{-1/\alpha} d_a(N)\big| < N_1^{-\eps} } \Big]
\le N^{-\alpha\eps} \int_{N^{-\eps}}^{N_1^{-\eps}} t^{1-\alpha } \, dt \le C N^{ -2\eps}.
\end{align}
This completes the proof.
\end{proof}

\begin{lemma}\label{l:newapprox2}
Fix $\eps >0$. Then there exists $C(\eps) > 0$ such that the following holds.
For every $1\le N_1, N_2 \le N$ such that $N/3 \le N_1 \le 2N/3$ and $N_1 + N_2 =N$, we have
\begin{align}\label{m1swap2}
\Big|
\E \big[ \log Z^\circ_{N_1} ( S_{N_1} ,N_1 )\big]
-
 \E \big[\log Z_{N_1}^\circ (S^{(1)},N ) \big]
 \Big|
 &\le C N^{1+ \epsilon(\alpha -2 )}\\
\Big|
\E \big[\log   Z^\circ_{N_2} ( S_{N_2} ,N_2 )\big]
-
 \E \big[\log Z_{N_2}^\circ (S^{(2)},N ) \big]
 \Big|
 &\le C N^{1+ \epsilon(\alpha -2 )}.\label{m2swap2}
\end{align}
\end{lemma}
\begin{proof}
We prove only \eqref{m1swap2}, since the proof of \eqref{m2swap2} is similar. 
As a first step towards \eqref{m1swap2}, we claim 
\begin{equation}\label{e:err2}
\Big| \E \big[\log Z_{N_1} (S^{(1)} , N)\big] - \E\big[ \log \hat Z^\circ_N \big] \Big| \le C N^{1+ \epsilon(\alpha -2 )},
\end{equation} 
where we recall that $\hat Z^\circ_N$ was defined after \eqref{hatHcirc2}. 
To this end, note that  $\E\langle\hat p^\circ\rangle = 0$ by integrating first in the $x_a$ variables, and that
\[
\E_x [e^{\hat p^\circ(\sigma)}] = \E_x \left[ \prod_{ 1 \le a \le S^{(1)} } \exp\big(  (1 - L^\circ_a) x_a \big)\right]
\le \exp(S^{(1)} N^{-2\eps}  ) \le \exp(3
 N^{-2\eps} S_N) \le 
 \exp(C N^{1 + \eps(\alpha-2)}),
\]
where we used the independence of the $(1 - L_a) x_a$ variables for the equality, and \Cref{l:expmom2} for the first inequality.
Thus by Lemma~\ref{lem:jensen} with $x=\widehat H^\circ$ and $y=\hat p^\circ$
as in \eqref{hatHcirc}, we obtain \eqref{e:err2}. %To this end, note that  $\E\langle\hat p^\circ\rangle = 0$ by integrating first in $x$.

Next, we claim that
\begin{equation}\label{e:err3}
\Big|  \E\big[ \log Z^\circ_{N_1} ( S_{N_1} ,N_1 )\big] - \E\big[ \log \hat Z^\circ_N \big] \Big| \le C N^{1+ \epsilon(\alpha -2 )}.
\end{equation} 
Together with \eqref{e:err2}, the previous equation implies the desired conclusion \eqref{m1swap2}.%, as discussed previously. 

To prove \eqref{e:err3},
we begin by identifying the distribution of $N_1^{-1/\alpha} d_a(N_1)$, the (rescaled) coupling distribution for the Hamiltonian $H^\circ _{S_{N_1}, N_1 , N_1}$.
%Let $X_a = N^{-1/\alpha} d_a(N)$ denote the (rescaled) coupling distribution for the Hamiltonian $H^\circ _{S^{(1)}, N_1 , N }$.
Using the definition of $d_a$ from \eqref{e:defd}, we have 
\[
\P\Big(\big|N_1^{-1/\alpha} d_a(N_1)\big| > t \Big)   =  \frac{1}{N_1^{\alpha\eps} t^\alpha},
\]
for $t \ge N_1^{-\eps}$.
%Let $\{Y_a\}_{a=1}^N$ be independent, identically distributed random variables such that the distribution of $Y_a$ is equal to the the conditional distribution of $X_a$ after conditioning on the event  $\{|X_a| > N_1^{-\eps} \}$. We have
%\begin{equation}\label{e:p1}
%\P\big(|X_a| > N_1^{-\eps}\big)  = \frac{1}{N^{\alpha\eps} N_1^{-\alpha\eps}} \frac{U(N^{1/\alpha} N_1^{-\eps})}{U(N^{1/\alpha -\eps})} = \frac{1}{N^{\alpha\eps} N_1^{-\alpha\eps}},
%\end{equation}
Similarly, we obtain $\P\big(|y_a| > t \big)  = 1/(N_1^{\alpha\eps} t^\alpha)$
for $t \ge N_1^{-\eps}$. Hence, the variables $N_1^{-1/\alpha} d_a(N_1)$ and $y_a$ are identically distributed, and we have the distributional equality
\begin{equation}\label{disteq}
H^\circ _{S_{N_1}, N_1 , N_1} \eqdist \sum_{1 \le a \le S_{N_1}} y_a \sigma_{c_a(1)} \sigma_{c_a(2)},
\end{equation}
where we recall that the $c_a$ variables are sampled uniformly from the set $\{ ( i,j) \}_{1 \le i \le j \le N_1}$.

Now note that the definition $\hat H^\circ$ in \eqref{lemon} differs from \eqref{disteq} only in the number of nonzero couplings $y_a$ (given by the indices $a$ such that $L_a=1$). 
There are $S^{(1)}$ nonzero couplings in $H^\circ _{S^{(1)}, N_1 , N }$, 
and the number nonzero couplings in $\hat H^\circ$ is binomial with $S^{(1)}$ trials and success probability $p^{(1)}$. The expectation of this distribution is 
\begin{equation}\label{e:binmean}
\frac{N_1^{\alpha\eps} }{N^{\alpha\eps}} S^{(1)} =  C_0 \cdot \frac{N_1^{1+\alpha\eps} }{2} + O(1)
=
 S_{N_1} + O(1).
%= \frac{U(N^{1/\alpha} N_1^{-\eps}) }{U(N_1^{1/\alpha - \eps})}p_{N_1} N_1^2
  \end{equation}
%edges with coupling greater than $N_1^{-\eps}$ in absolute value. 
Then an argument nearly identical to the one that proved \eqref{m1swap} shows \eqref{e:err3}. This completes the proof.
\end{proof}

\begin{proposition}\label{p:2}
There exists $c> 0$ such that the following holds for all $\eps \in (0,c)$. We have 
$\lim_{N\rightarrow \infty} F^\circ_{S,N,N} = L$
for some $L$ satisfying $ -\infty < L \le \infty$. 
\end{proposition}

\begin{proof}
Let $1\le N_1,N_2\le N$ be integers such that $N/3 \le N_1 \le 2N/3$. 
\Cref{l:main} and \Cref{l:newapprox} together imply that
\[\E \big[\log Z_N^\circ (S,N )  \big] + CN^{2/3} \ge \E \big[\log Z_{N_1}^\circ (S^{(1)},
N) \big] + \E \big[\log Z^\circ_{N_2} ( S^{(2)}, N)\big],
\]
if $\eps$ in chosen small enough (relative to $\alpha$). Then \Cref{l:newapprox2} implies that 
\[
\E \big[\log Z_N^\circ (S , N)  \big] + CN^{1 + \eps(\alpha-2)}\ge \E \big[\log Z^\circ_{N_1} ( S_{N_1} ,N_1 ) \big] + \E \big[\log Z_{N_2}^\circ (S_{N_2},N_2 )\big].
\]
Now set
$a_N =  - \E \big[\log Z_N^\circ (S,N )  \big]$ and $\phi(t) = C t^{1+ \epsilon(\alpha -2 )}$,
and  observe that $\phi(t)/t^2$ is integrable on $[1, \infty)$ since $1 + \eps(\alpha - 2) < 1$. We then apply \Cref{l:subadd}  to conclude.
\end{proof}

\begin{proof}[Proof of \Cref{t:main}]

By combining \Cref{l:l2}, \Cref{l:fixededge}, \Cref{l:multimodel}, and \Cref{p:2}, we find that $ \lim_{N\rightarrow \infty} F_N$ exists and 
$\lim_{N\rightarrow \infty} F_N >  -\infty.$
It remains to show that this limit does not equal $+\infty$. To accomplish this, we will show that $F_N$ is uniformly bounded. We consider the Hamiltonian $\hat H_\star(\sigma)$ defined by
\[
\widehat H_\star(\sigma) = \frac{1}{N^{1/\alpha}} \sum J_{ij} \one_{| J_{ij} | \ge  R_\star }\sigma_i \sigma_j,\quad  \hat p_\star(\sigma) = \frac{1}{N^{1/\alpha}} \sum J_{ij}  \one_{| J_{ij} | <  R_\star }\sigma_i \sigma_j,\quad R_\star = N^{1/\alpha}.
\]
We note that
$H(\sigma) = \widehat H_\star(\sigma)  + \hat p_\star(\sigma),$
and define $Z_{N,\star}$ and $F_{N,\star}$ by analogy with \eqref{e:perturb}.

The second inequality in Lemma~\ref{lem:jensen} yields % \eqref{e:perturbation} with the choice $\eps=0$ yields (since $R_\star$ equals $R$ with the choice $\eps =0$)
\begin{equation}\label{e:perturbation11}
\E [ \log Z_{N,\star} ]  = \E \big[\log \textstyle\sum_\sigma e^{\widehat H_\star(\sigma)  + \hat p_\star(\sigma) } \big] \le \E \big[\log \sum_\sigma e^{\widehat H_\star(\sigma) }\big] + O(N).
\end{equation}

Therefore, it suffices to bound the expectation on the right side of \eqref{e:perturbation11}.  
\begin{comment}The number of nonzero couplings $G_{ij} = J_{ij} \one_{|J_{ij}| > R_\star}$ in the Hamiltonian $H_\star$ which are nonzero is distributed as $\mathcal R = \Bi( N(N+1)/2, t_N)$,
with 
\begin{equation}
t_N  = \P( |J| > N^{-1/\alpha}) \le C N^{-1},
\end{equation}
by \eqref{e:arep}.
\end{comment}
Define the Hamiltonian $H_0$ by $H_0(\sigma) =0$. Its associated partition function is $Z_0=2^N$. Then removing all the nonzero couplings of $H_\star$ using \eqref{e:edgeremoval} with $\eps = 0$ gives
\begin{equation}\label{e:edgeremoval2}
\big| \log  Z_0  - \log Z_{N,\star} \big| \le N^{-1/\alpha}\sum_{1\le i<j \le N} |J_{ij}| \one_{|J_{ij}| > R_\star},
\end{equation}
which implies
\begin{equation}\label{e:edgeremoval3}
\Big| \E \big[ \log  Z_0\big]  - \E \big[\log Z_{\star}\big] \Big| \le 
C N^{2 - 1/\alpha} \cdot \E\Big[ |J_{ij}| \one_{|J_{ij}| > R_\star} \Big] = O(N),
\end{equation}
where we used \eqref{e:arep} to compute 
$\E\Big[ |J_{ij}| \one_{|J_{ij}| > R_\star} \Big] \le N^{-1 + 1/\alpha}.$
Since $|\log Z_0 |\le CN $, equation \eqref{e:edgeremoval3} implies 
$N^{-1} \E \big[\log Z_{\star,N}\big] \le C.$
Combining this bound with \eqref{e:perturbation11} completes the proof.
\end{proof}
\begin{comment}
\begin{proof}[Proof of \Cref{c:linear}]
For brevity, we set $G_N = \max_{\sigma \in \Sigma_N} H(\sigma)$. Then the definition of $Z_N$ implies that 
\begin{equation}
G_N \le \log (Z_N) \le N \log 2 + G_N,
\end{equation}
and hence
\begin{equation}
0\le  N^{-1} \E \big[ \log (Z_N)\big] - N^{-1} \E[G_N]  \le  \log 2. 
\end{equation}
\end{proof}
\end{comment}
\section{Interpolation}\label{s:lemma4}
In this section, we prove \Cref{l:main}.

\subsection{Proof of \Cref{l:main}}
Recall the notation of \Cref{s:multiedge}. Given integers $v, u >0$,  define $\G(v, u)$ to be the random  multi-graph on the vertex set $\unn{1}{v}$ with edge set $\{ c_a \}_{1 \le a \le u}$. %To each edge $c_a$,  $N^{-1/\alpha} d_a (N)$.
We will construct a sequence of multi-graphs interpolating between $\G(N, S_N)$ and the disjoint union of $\G(N_1, \mathcal M_2)$ and $\G(N_2, \mathcal M_2)$. 

Given $N_1, N_2$ such that $N_1 + N_2$ and an integer $r$ such that $0\le r \le N$, we define $\G_r$ as follows. Let $\chi$ be a Bernoulli random value that takes the value $1$ with probability $N_1/N$, and is $0$ otherwise, and let $\{ \chi_a \}_{1 \le a \le N}$ be a collection of independent random variables distributed as $\chi$. Let $\{c^{(1)}_a \}_{1 \le a \le S}$ be independent edges chosen uniformly at random from the set $\{(i,j) \}_{1 \le i \le j \le N_1}$, and define $\{c^{(2)}_a \}_{1 \le a \le S}$ similarly for $\{(i,j) \}_{N_1 + 1 \le i \le j \le N_2}$. We define the random variables $\{c^{(-)}_a \}_{1 \le a \le S}$ by letting $c^{(-)}_a = c^{(1)}_a $ if $\chi =1$, and $c^{(-)}_a = c^{(2)}_a $  if $\chi = 0$.
%\begin{equation}
%c^{(-)}_a  = \chi c^{(1)}_a + ( 1 - \chi)c^{(2)}_a.
%\end{equation}
\begin{comment}
In words, $c^{(-)}$ first samples from a Bernoulli random variable with parameter $N_1/N$. If the result is $1$, $e_j$ is sampled uniformly at random from all possible edges between the vertices labeled $1$ to $N_1$. If the result is $0$, $e_j$ is sampled uniformly at random from all possible edges between the $N_2$ total vertices labeled $N_1 + 1$ to $N$. 
\end{comment}
%The weight for $c^{(+)}_a$ is again given by $N^{-1/\alpha} d_a(N)$. 
The graph $\G_r$ is then defined for $0\le r \le S_N$ by the random edge set $\{ c_a \}_{1 \le a \le r} \cup \{ c^{(-)}_a \}_{r+1 \le a \le S}$. We see that the graphs $\G_r$ interpolate between $\G(N, S)$ when $r=S_N$ and the disjoint union of $\G(N_1, \mathcal M_2)$ and $\G(N_2, \mathcal M_2)$ when $r=0$.
\begin{comment}
The graph $\G_r$ may contain loops and multiple edges. The first $r$ edges $e_1, \dots , e_r$ are chosen uniformly at random from all possible edges $\{(i,j) : 1 \le i ,j \le N\}$. The remaining edges are sampled using a two-step procedure. First sample from a Bernoulli random variables with parameter $N_1/N$. If the result is $1$, $e_j$ is sampled uniformly at random from all possible edges between the vertices labeled $1$ to $N_1$. If the result is $0$, $e_j$ is sampled uniformly at random from all possible edges between the $N_2$ total vertices labeled $N_1 + 1$ to $N$. The couplings for the edges are chosen as before, from $N^{-1/\alpha} J$ conditional on $|J| > N^{1/\alpha - \epsilon}$. Checking the cases $r=S$ and $r=0$, we see this gives the claimed interpolation.
\ber define in terms of $c^{(r)}_a$.\eer
\end{comment}

We define a Hamiltonian and partition function corresponding to $\G_r$ by 
\[
H^{(r)}(\sigma) = N^{-1/\alpha} \sum_{1 \le a \le r} d_a(N) \sigma_{c_a(i)} \sigma_{c_a(j)} + N^{-1/\alpha} \sum_{r+1 \le a \le S_N} d_a(N) \sigma_{c_a^{(-)}(1)} \sigma_{c_a^{(-)}(2)}
\]
and
$Z^{(r)} = Z^{(r)}_N = \sum_{\sigma\in \Sigma_N} \exp\big( H^{(r)}(\sigma) \big).$
We also define the graph $\widehat \G_r$ using the random edge set  $\{ c_a \}_{1 \le a \le r-1} \cup \{ c^{(-)}_a \}_{r+1 \le a \le S}$, which omits the $r$-th edge. 
The corresponding Hamiltonian and partition function are defined by
\begin{align*}
H^{(r, -)}(\sigma) &= N^{-1/\alpha} \sum_{1 \le a \le r-1} d_a(N) \sigma_{c_a(i)} \sigma_{c_a(j)} + N^{-1/\alpha} \sum_{r+1 \le a \le S_N} d_a(N) \sigma_{c_a^{(-)}(1)} \sigma_{c_a^{(-)}(2)}, \\
Z^{(r,-)} &= Z^{(r,-)}_N = \sum_{\sigma\in \Sigma_N} \exp\big( H^{(r,-)}(\sigma) \big).
\end{align*}

\begin{lemma}\label{l:interpolate} For every $1\le r  \le S_N$, 
\begin{equation}\label{interpolationsuffices}
\E \big[\log Z^{(r)} \big]\ge \E 
\big[\log Z^{(r-1)}\big]. \end{equation}
\end{lemma}
\begin{proof}
%Fix $r$. The graph $\G_{r-1}$ is obtained from $\G_r$ by deleting an edge chosen uniformly at random from $e_1$ to $e_r$ and then adding an edge using the two-step sampling procedure. Let $\widehat \G_r$ be the graph after deleting but before adding the edge. 
It suffices to show that
\begin{equation}\label{interpolationsuffices2}
\E\big[  \log Z^{(r)} \;\big|\; \widehat \G_r \big]  
\ge  \E\big[  \log Z^{(r-1)} \; \big| \; \widehat \G_r  \big],
\end{equation}
where the notation in the previous inequality denotes the conditional expectation over the edges and weights of $\widehat \G_r$. The remaining randomness is in the choice of edge $c_r$ (or $c^{(-)}_r$) and the weight $d_r$. We write $x = c_r (1)$ and $y = c_r(2)$. 

We compute
\begin{align}
&\E\big[  \log Z^{(r)}  \;|\: \widehat \G_r \big] - \E \big[ \log Z^{(r,-)} \big]\nonumber\\% &= \E\left[ \log   \frac{Z^{(r)}  }{Z^{(r, -)}  } \;\Bigg|\; \widehat \G_r \right] \\
&= \E\left[ \log \frac{ e^{-d_r} \sum_\sigma  \one_{\sigma_x \neq \sigma_y} \exp\big(H^{(r,-)}(\sigma)\big) + e^{d_r} \sum_\sigma \one_{\sigma_x = \sigma_y} \exp\big(H^{(r,-)}(\sigma)\big)}{\sum_\sigma \exp\big(H^{(r,-)}(\sigma)\big)} \;\Bigg|\; \widehat \G_r \right].
\label{e:continue}
\end{align}
The same expression holds with $Z^{(r)}$ replaced by $Z^{(r-1)}$, and $x$ and $y$ replaced by $x^{(-)}=c_r^{(-1)}(i)$ and $y^{(-)}=c_r^{(-)}(j)$, respectively.
In the following two cases, we will compute both of these expressions, after conditioning on  $d_r$. The computations will differ depending on the sign of $d_r$. \\
%We condition on the sign of $d_r$ and drop this conditioning from the notation.\\
\paragraph{\bf Case I: $d_r<0$.} Let $\mu$ denote the Gibbs measure for the Hamiltonian $H^{(r,-)}(\sigma)$. Using \eqref{e:continue}, we have
\begin{align*}
&\E[  \log Z^{(r)}  \;|\: \widehat \G_r, d_r   ] - \E [ \log Z^{(r,-)}  \;|\:  d_r ]\\ 
&= - d_r  +  \E\left[ \log \frac{ \sum_\sigma  \one_{\sigma_x \neq \sigma_y} \exp\big(H^{(r,-)}(\sigma)\big) + e^{2d_r} \sum_\sigma \one_{\sigma_x= \sigma_y} \exp\big(H^{(r,-)}(\sigma)\big)}{\sum_\sigma \exp\big(H^{(r,-)}(\sigma)\big)} \;\Bigg|\; \widehat \G_r \right] \\ 
%&= - d_r + \E\left[ \log\left( 1 + (e^{2d_r} -1) \mu(\sigma_x = \sigma_y)\right) \;\Bigg|\;\widehat \G_r, d_r  \right]  \\ 
&= - d_r + \E\left[ \log\left( 1 - (1 - e^{2d_r}) \mu(\sigma_x= \sigma_y )\right) \;\Bigg|\; \widehat \G_r , d_r\right] .
\end{align*}
Observe that $0 < (1 - e^{2d_r}) \mu(\sigma_x = \sigma_y )< 1$ because $d_r< 0$, so it is permissible to Taylor expand the logarithm. Therefore, introducing replicas $\sigma^\ell$, we have 
\begin{align}%\label{514}
&\E\big[  \log Z^{(r)}  \;|\: \widehat \G_r, d_r  \big] - \E \big[ \log Z^{(r,-)}  \;|\:  d_r\big]  + d_r \nonumber \\
&= - \sum_{k=1}^\infty  \E \left[\frac{(1 - e^{2d_r})^k \mu(\sigma_x = \sigma_y)^k}{k } \;\Bigg|\; \widehat \G_r, d_r \right]\nonumber\\ 
&=  - \sum_{k=1}^\infty \frac{(1 - e^{2d_r})^k}{k}  \E\left[  \sum_{\sigma^1, \dots, \sigma^k}  \frac{ \exp\left( \sum_{\ell=1}^k H^{(r,-)}(\sigma^\ell)  \right)}{(Z^{(r,-)})^k} \one_{ \{\sigma^\ell_x = \sigma^\ell_y ,\forall \ell \}} \;\Bigg|\; \widehat \G_r, d_r \right]\nonumber\\ 
&=  - \sum_{k=1}^\infty \frac{(1 - e^{2d_r})^k}{k}  \sum_{\sigma^1, \dots, \sigma^k}  \frac{ \exp\left( \sum_{\ell=1}^k H^{(r,-)}(\sigma^\ell)  \right)}{(Z^{(r,-)})^k}  \cdot \E\left[  \one_{ \{\sigma^\ell_x = \sigma^\ell_y ,\forall \ell \}} \right].
\label{e:prev}
\end{align}

For every set of replicas $\boldsymbol{\sigma} = (\sigma^1,\dots \sigma^k)$, we introduce the following equivalence relation on $\unn{1}{N}$. For $i,j \in \unn{1}{N}$, we say that $i \sim j$ if $\sigma^\ell_i = \sigma^\ell_j$ for all replicas $\ell = 1, \dots k$. Denote the number of equivalence classes induced by $\sim$ by $J$, and let $\{O_s\}_{s=1}^J = \{O_s(\boldsymbol{\sigma})\}_{s=1}^J$ be the set of these equivalence classes. Recalling the definition of $x$ and $y$, we compute 
$\E[  \one_{ \{\sigma^\ell_x = \sigma^\ell_y ,\forall \ell \}}  ] =  \sum_{s=1}^J ( |O_s|/N)^2.$
Then, recalling \eqref{e:prev}, we have
\begin{align}
&\E\big[  \log Z^{(r)}  \;|\: \widehat \G_r, d_r  \big] - \E \big[ \log Z^{(r,-)}  \;|\:  d_r\big]  +  d_r\nonumber\\ &= - \sum_{k=1}^\infty \frac{(1 - e^{2d_r})^k}{k}  \sum_{\sigma^1, \dots, \sigma^k}  \frac{ \exp\left( \sum_{\ell=1}^k H^{(r,-)}(\sigma^\ell)  \right)}{(Z^{(r,-)})^k} \sum_{s=1}^J \left( \frac{|O_s|}{N}\right)^2.\label{mango1}
\end{align}

The computation for $\E\big[  \log Z^{(r-1)}  \;|\: \widehat \G_r, d_r  \big] - \E \big[ \log Z^{(r,-)}  \;|\:  d_r\big]$ is analogous, and we now outline the main steps. Note that in $\G_{r-1}$, the $r$-th edge is added using the two-step sampling procedure described at the beginning of this proof, where first the value $\chi_r$ is sampled, and then $c_r^{(-)}$ is sampled from either $c_r^{(1)}$ or $c_r^{(2)}$, depending on the value of $\chi_r$. Recall that $\big(x^{(-)},y^{(-)}\big)$ denotes the random edge $c^{(-)}_r$. We find
\begin{equation}
\E[  \one_{ \{\sigma^\ell_{x^{(-)}} = \sigma^\ell_{y^{(-)}} ,\forall \ell \}}]=
\sum_{s=1}^J \left( \frac{N_1}{N} \left(\frac{| O_s \cap \unn{1}{N_1}|}{N_1}\right)^2 + \frac{N_2}{N}\left( \frac{| O_s \cap \unn{1}{N_2}|}{N_2}\right)^2 \right).
\end{equation} 
Then the analogue of \eqref{e:prev} holds for $\E\big[  \log Z^{(r-1)}  \;|\: \widehat \G_r, d_r  \big] - \E \big[ \log Z^{(r,-)}  \;|\:  d_r\big]$, with $x$ and $y$ replaced by $x^{(-)}$ and $y^{(-)}$, respectively, and we conclude that
\begin{multline}
\E[  \log Z^{(r-1)}  \;|\: \widehat \G_r, d_r  ] - \E [ \log Z^{(r,-)}  \;|\:  d_r]  + d_r= \\
- \sum_{k=1}^\infty \frac{(1 - e^{2d_r})^k}{k}  \sum_{\sigma^1, \dots, \sigma^k}  \frac{ \exp\left( \sum_{\ell=1}^k H^{(r,-)}(\sigma^\ell)  \right)}{(Z^{(r,-)})^k} \\ \times \sum_{s=1}^J \left( \frac{N_1}{N} \left(\frac{| O_s \cap \unn{1}{N_1}|}{N_1}\right)^2 + \frac{N_2}{N}\left( \frac{| O_s \cap \unn{1}{N_2}|}{N_2}\right)^2 \right) .\label{mango2}
\end{multline}
% The sum $\G_r$ is greater than that for $\G_{r-1}$ after using the convexity of $\sigma^2$ to compare the sums term by term.

\paragraph{\bf Case II: $d_r> 0$.} We proceed as in the previous case to obtain
\begin{align*}
\E[  \log Z^{(r)}  \;|\: \widehat \G_r, d_r  ] - \E [ \log Z^{(r,-)}  \;|\:  d_r]  + d_r 
%&=    \E\left[ \log \frac{ e^{- 2d_r} \sum_\sigma  \one_{\sigma_x \neq \sigma_y} \exp\big(H^{(r,-)}(\sigma)\big) + \sum_\sigma \one_{\sigma_x = \sigma_y} \exp\big(H^{(r,-)}(\sigma)\big)}{\sum_\sigma \exp\big(H^{(r,-)}(\sigma)\big)} \;\Bigg|\; \widehat \G_r, d_r \right] \\ 
=   \E\Big[ \log\big( 1  - (1 - e^{-2d_r} ) \mu(\sigma_x \neq \sigma_y )\big)\;|\; \widehat \G_r, d_r \Big].
\end{align*}
Taylor expanding the logarithm, we obtain as before that
\begin{align}
&\E\big[  \log Z^{(r)}  \;|\: \widehat \G_r, d_r  \big] - \E \big[ \log Z^{(r,-)}  \;|\:  d_r\big]  + d_r\nonumber\\ 
%&= - \sum_{k=1}^\infty  \E \left[\frac{(1 - e^{-2d_r})^k \mu(\sigma_x \neq \sigma_y)^k}{k } \;\bigg|\; \widehat \G_r , d_r \right]
%\\ &=  - \sum_{k=1}^\infty \frac{(1 - e^{-2d_r})^k}{k}  \E\left[  \sum_{\sigma^1, \dots, \sigma^k}  \frac{ \exp\left( \sum_{\ell=1}^k H^{(r,-)}(\sigma^\ell)  \right)}{(Z^{(r,-)})^k} \one_{ \{\sigma^\ell_x \neq \sigma^\ell_y ,\forall \ell \}} \;\Bigg|\; \widehat \G_r, d_r \right]\\ 
&=  - \sum_{k=1}^\infty \frac{(1 - e^{-2d_r})^k}{k}  \sum_{\sigma^1, \dots, \sigma^k}  \frac{ \exp( \sum_{\ell=1}^k H^{(r,-)}(\sigma^\ell)  )}{(Z^{(r,-)})^k}  \E[  \one_{ \{\sigma^\ell_x \neq \sigma^\ell_y ,\forall \ell \}} ].\label{mango3}
\end{align}

We now compute the term $\E[  \one_{ \{\sigma^\ell_x \neq \sigma^\ell_y ,\forall \ell \}} ]$. 
We recall the equivalence classes $O_s$ defined in the previous case. For every class $O_s$, there exists an equivalence class $O_r$, for some $r=r(s)$, of vertices such that $\sigma_i^{\ell} \neq \sigma_j^{\ell}$ for all $\ell$ if $i \in O_s$ and $j \in O_r$. This gives a pairing of equivalence classes. Then we have
$\E[  \one_{ \{\sigma^\ell_x \neq \sigma^\ell_y ,\forall \ell \}}  ] = N^{-2} \sum_{s=1}^J |O_s| |O_r|$ which combined with  %\label{mango4}
%Then combining 
\eqref{mango3} %and \eqref{mango4}, we have 
yields
\begin{align}
&\E\big[  \log Z^{(r)}  \;|\: \widehat \G_r, d_r  \big] - \E \big[ \log Z^{(r,-)}  \;|\:  d_r\big]  + d_r\nonumber\\ &=
 - \sum_{k=1}^\infty \frac{(1 - e^{-2d_r})^k}{k}  \sum_{\sigma^1, \dots, \sigma^k}  \frac{ \exp\left( \sum_{\ell=1}^k H^{(r,-)}(\sigma^\ell)  \right)}{(Z^{(r,-)})^k} \sum_{s=1}^J \left( \frac{|O_s|}{N}\right) \left( \frac{|O_r|}{N}\right).\label{mango6}
 \end{align}
Similarly, we compute
\begin{align*}
\E[  \one_{ \{\sigma^\ell_{x^{(-)}} = \sigma^\ell_{y^{(-)}} ,\forall \ell \}} ] &= \frac{N_1}{N}  \sum_{s=1}^J   \frac{| O_s \cap \unn{1}{N_1}|}{N_1}\cdot\frac{| O_r \cap \unn{1}{N_1}|}{N_1}
+ \frac{N_2}{N}\sum_{s=1}^J  \frac{| O_s \cap \unn{1}{N_2}|}{N_2} \cdot \frac{| O_r \cap \unn{1}{N_2}|}{N_2},\label{mango5}
\end{align*}
leading to
\begin{equation}
\begin{aligned}\label{mango7}
&\E\big[  \log Z^{(r-1)}  \;|\: \widehat \G_r, d_r  \big] - \E \big[ \log Z^{(r,-)}  \;|\:  d_r\big]  + d_r\\ &=
 - \sum_{k=1}^\infty \frac{(1 - e^{-2d_r})^k}{k}  \sum_{\sigma^1, \dots, \sigma^k}  \frac{ \exp\left( \sum_{\ell=1}^k H^{(r,-)}(\sigma^\ell)  \right)}{(Z^{(r,-)})^k}\\
 &\times 
 \Big(
 \frac{N_1}{N}  \sum_{s=1}^J   \frac{| O_s \cap \unn{1}{N_1}|}{N_1}\cdot\frac{| O_r \cap \unn{1}{N_1}|}{N_1}+ \frac{N_2}{N}\sum_{s=1}^J  \frac{| O_s \cap \unn{1}{N_2}|}{N_2}\cdot \frac{| O_r \cap \unn{1}{N_2}|}{N_2}\Big).
\end{aligned}
\end{equation}
\\

\paragraph{\bf Conclusion.} %Fix some weight $\hat w >0$. We choose the coupling to be $\hat d_r$ or $-\hat d_r$ with equal probability.

 Observe that $1 - e^{ 2 ( - x)} = 1 - e^{-2x}$, so the powers $(1 - e^{2d_r})^k$ in the Taylor expansions in above two cases are the same if $|d_r| =x$ in each case. Further, observe that the density of $d_r$ is symmetric, by definition. 
We now subtract \eqref{mango2} from \eqref{mango1},  subtract \eqref{mango7} from \eqref{mango6}, and take expectation over $d_r$ and $\widehat{\mathbb{G}}_r$ in each expression.
The upshot of this computation is that to establish \eqref{interpolationsuffices2}, it suffices to prove for a fixed replica $\boldsymbol{\sigma}$ that
\begin{align*}
&\sum_{s=1}^J \left( \frac{N_1}{N} \left(\frac{| O_s \cap \unn{1}{N_1}|}{N_1}\right)^2 + \frac{N_2}{N}\left( \frac{| O_s \cap \unn{1}{N_2}|}{N_2}\right)^2 \right) 
\\ &+ \sum_{s=1}^J \left( \frac{N_1}{N} \left(\frac{| O_s \cap \unn{1}{N_1}|}{N_1}\right) \left(\frac{| O_r \cap \unn{1}{N_1}|}{N_1}\right)+ \frac{N_2}{N}\left( \frac{| O_s \cap \unn{1}{N_2}|}{N_2}\right) \left( \frac{| O_r \cap \unn{1}{N_2}|}{N_2}\right)   \right)
\\ &\quad \ge   \sum_{s=1}^J \left( \frac{|O_s|}{N}\right)^2 +  \sum_{s=1}^J \left( \frac{|O_s|}{N}\right) \left( \frac{|O_r|}{N}\right).
\end{align*} 
Fix some replica $s$ and corresponding $r=r(s)$ (as defined in the second case above), and consider just these terms in the sum. It suffices to show that
\begin{align*}
&\left( \frac{N_1}{N} \left(\frac{| O_s \cap \unn{1}{N_1}|}{N_1}\right)^2 + \frac{N_2}{N}\left( \frac{| O_s \cap \unn{1}{N_2}|}{N_2}\right)^2 \right)\\ &+ \left( \frac{N_1}{N} \left(\frac{| O_r \cap \unn{1}{N_1}|}{N_1}\right)^2 + \frac{N_2}{N}\left( \frac{| O_r \cap \unn{1}{N_2}|}{N_2}\right)^2 \right) 
\\ &+ 2  \left( \frac{N_1}{N} \left(\frac{| O_s \cap \unn{1}{N_1}|}{N_1}\right) \left(\frac{| O_r \cap \unn{1}{N_1}|}{N_1}\right)+ \frac{N_2}{N}\left( \frac{| O_s \cap \unn{1}{N_2}|}{N_2}\right) \left( \frac{| O_r \cap \unn{1}{N_2}|}{N_2}\right)   \right)
\\ &\quad \ge   \left( \frac{|O_r|}{N}\right)^2  +  \left( \frac{|O_s|}{N}\right)^2 +  2 \left( \frac{|O_s|}{N}\right) \left( \frac{|O_r|}{N}\right).
\end{align*} 
The right side of the previous inequality factors as 
$\left(  \frac{|O_s|}{N} + \frac{|O_r|}{N}  \right)^2$, 
whereas the left side factors as 
\begin{equation*}%\label{left1}
\frac{N_1}{N}\left( \frac{| O_s \cap \unn{1}{N_1}|}{N_1}  + \frac{| O_r \cap \unn{1}{N_1}|}{N_1} \right)^2 + \frac{N_2}{N}\left( \frac{| O_s \cap \unn{1}{N_2}|}{N_2}  + \frac{| O_r \cap \unn{1}{N_2}|}{N_2} \right)^2.
\end{equation*}
The left side is thus greater than the right side by the convexity of $x \mapsto x^2$. This establishes \eqref{interpolationsuffices} and completes the proof.
\end{proof}

\begin{proof}[Proof of \Cref{l:main}]
We apply \Cref{l:interpolate} in succession for $r=1, \dots , S$ to obtain
\begin{equation}\label{e:finalclaim}
\E \big[\log Z^{(N)}\big] \ge \E \big[\log Z^{(0)}\big].
\end{equation}
By the definition of $Z^{(r)}$, equation \eqref{e:finalclaim} is exactly the claim \eqref{e:main}.
\end{proof}

\section{Self-Averaging}\label{s:concentration}
\begin{proof}[Proof of \Cref{t:concentration}]
This follows from the following Proposition together with Markov's inequality.
\end{proof}
\begin{proposition}\label{p:martingale}
We have
\[
 \E \bigg[
\Big| \log Z_N
- \E\big[ \log Z_N  \big] \Big|^2 \bigg] \le N^{3 - \alpha +\delta}.
\]
\end{proposition}
\begin{proof}
Let $\mathcal I = \big\{(i,j) : 1\le i < j \le N \big\}$ and fix an arbitrary bijection $f\colon \mathcal I \rightarrow \big\{1,2,\dots, | \mathcal I |\big\}$. We use $J_{x}$ with $x \in \big\{1,2,\dots, | \mathcal I |\big\}$ as shorthand for $J_{f^{-1}(x)}$. 
Set $\mathcal F_x = \sigma (J_y : y \le x )$ for all $x \in \big\{1,2,\dots, | \mathcal I |\big\}$, where this notation denotes the $\sigma$-algebra generated by the given couplings $J_y$. 
Consider the martingale
\begin{equation}
A_x = \frac{1}{N} \E\big[ \log Z_N \;\big|\; \mathcal F_x\big]
- \frac{1}{N} \E\big[ \log Z_N \big],
\end{equation}
with the convention that $A_0 = 0$. 

Define the martingale difference sequence $D_x = A_x - A_{x-1}$ for $x \ge 1$, so that 
$A_x = \sum_{y\le x} D_y.$
%We now estimate $\E \big[ D_x^2 \big]$. 
Set 
\[H^{(x)}(\sigma) =  \frac{1}{N^{1/\alpha}} \sum_{i < j } J_{ij} \sigma_i \sigma_j \one_{f(i,j) \neq x},\qquad Z^{(x)}_N = \sum_{\sigma\in \Sigma_N} e^{  H^{(x)}(\sigma)},
\]
where $H^{(x)}(\sigma)$ is similar to the Hamiltonian $H(\sigma)$, except with the coupling $J_x$ set equal to zero. Let $\langle \cdot \rangle_x$ denote the Gibbs measure with respect to $H^{(x)}$. Then we have (by definition)
$Z_N = Z_N^{(x)} \langle e^{ N^{-1/\alpha}J_x \sigma_x} \rangle_x.$
We write
\[
N \cdot D_x = \E\big[\log \langle e^{N^{-1/\alpha} J_x \sigma_x} \rangle_x \;\big |\; \mathcal F_x \big]
-  \E\big[\log \langle e^{N^{-1/\alpha} J_x \sigma_x} \rangle_x \;\big |\; \mathcal F_{x-1} \big],
\]
where we use the fact that $\E \big[ Z^{(x)}\;\big|\; \mathcal F_x\big] = \E \big[ Z^{(x)}\;\big|\; \mathcal F_{x-1}\big]$. % is the same conditioned on either $\mathcal F_x$ for $\mathcal F_{x-1}$. 
Bounding $e^{N^{-1/\alpha} J_x \sigma_x}$ in absolute value in each expectation gives
\[
|D_x| \le N^{-1-1/\alpha} \Big( 
|J_x| + \E \big[ | J_x| \big] \Big),
\]
which implies 
$|D_x|^p \le 2^p N^{-1-1/\alpha} ( 
|J_x|^p + \E [ | J_x| ]^p )$
for any $p \in (1,2)$.
By Burkholder's inequality with exponent $p \in( 1,2)$, and the fact that $p/2 <1$, we have 
\[
\E[ 
|A_{|\mathcal I |} - A_0|^2
] \le C_p \E [ 
(\textstyle\sum_{x} D_x^2 )^{p/2}  
]
\le
C_p \E \big[ 
\textstyle\sum_{x} D_x^p 
\big]
\le C_p N^{2 -p -p/\alpha} \E \big[ | J_x| \big]^p.
\]
Set $g(p) = 2 - p\big(1 + \frac{1}{\alpha}\big)$. Observe that $g$ is continuous on $[1,2]$, and $g(\alpha) <  1 - \alpha < 0$. By choosing $p(\delta)$ sufficiently close to $\alpha$, we find 
\[
\E\big[ 
|A_{|\mathcal I |} - A_0|^2
\big] \le C N^{ 1 - \alpha + \delta},
\]
where $C=C(\delta) > 1$ depends on $\delta$. This completes the proof.
\end{proof}

{\small \bibliography{levyglass} }
	
\bibliographystyle{abbrv}
\end{document}